\documentclass[12pt]{amsart}
\usepackage[utf8]{inputenc}
\usepackage{amsmath} 
\usepackage{amssymb}
\usepackage{graphicx}
\usepackage{graphicx,color}
\usepackage{latexsym}
\usepackage[all]{xy}
\usepackage{mathrsfs}
\usepackage{enumerate}
\usepackage{amssymb}
\usepackage{tikz}
\usepackage{tikz-qtree}
\usepackage{breqn}
\usepackage{easyReview}
\usepackage{subfig}
\usepackage{hyperref}
\usetikzlibrary{shapes.geometric}

\makeatletter
\renewcommand{\subsection}{\@startsection
{subsection}{2}{0mm}{\baselineskip}{-0.25cm}
{\normalfont\normalsize\em}}
\makeatother

\def\negs{\mathbf s}
\def\negh{\mathbf h}
\def\negn{\mathbf n}

\def\negx{\mathbf x}
\def\negb{\mathbf b}
\def\nega{\mathbf a}
\def\negf{\mathbf f}
\def\negy{\mathbf y}
\def\negz{\mathbf z}
\def\nege{\mathbf e}

\def\negt{\mathbf t}

\def\neg1{\text{\boldmath$1$}}

\def\neg1{\text{\boldmath$1$}}

\def\neg0{\bf {0}}

\def\NN{\mathbb{N}}

\newtheorem{theorem}{Theorem}[section]
\newtheorem{proposition}[theorem]{Proposition}
\newtheorem{corollary}[theorem]{Corollary}
\newtheorem{lemma}[theorem]{Lemma}

{\theoremstyle{definition}
\newtheorem{definition}[theorem]{Definition}
\newtheorem{example}[theorem]{Example}

}

{\theoremstyle{remark}
\newtheorem{remark}[theorem]{Remark}
}

\title[On almost-symmetry in generalized numerical semigroups]{On almost-symmetry in \\ generalized numerical semigroups}

\author[C. Cisto]{Carmelo Cisto}
 \address{Università di Messina, Dipartimento di Scienze Matematiche e Informatiche, Scienze Fisiche e Scienze della Terra, Viale Ferdinando Stagno D’Alcontres 31, 98166 Messina, Italy}
  \email{ccisto@unime.it}

\author[W. Tenório]{Wanderson Tenório}
 \address{Universidade Federal de Goiás, Instituto de Matemática e Estatística, R. Jacarandá, Chácaras Califórnia, 74001-970 Goiânia, GO, Brazil}
  \email{wanderson\_tenorio@ufg.br}

\begin{document}
	
	\dedicatory{Dedicated to the memory of Fernando Torres.}


\maketitle
\begin{abstract} In this work we introduce the notion of almost-symmetry for generalized numerical semigroups. In addition to the main properties occurring in this new class, we present several characterizations for its elements. In particular we show that this class yields a new family of Frobenius generalized numerical semigroups and extends the class of irreducible generalized numerical semigroups. This investigation allows us to provide a method of computing all almost symmetric generalized numerical semigroup having a fixed Frobenius element and organizing them in a rooted tree depending on a chosen monomial order.
\end{abstract}

\section{Introduction}

Let $\mathbb{Z}$ and $\NN_0$ denote the sets of integers and nonnegative integers respectively. Given a positive integer $d$, we say that a subset $S\subseteq \NN_0^d$ is \emph{submonoid} of $\NN_0^d$ if it is closed under the usual addition in $\NN_0^d$ and contains $\mathbf{0}$, the element of $\NN_0^d$ with all zero components. A \emph{numerical semigroup} is a submonoid of $\NN_0$ (and so $d=1$) such that its complement in $\NN_0$ is finite. These objects have much importance in  many areas and are extensively studied in several directions. For a detailed exposition on this subject, we refer the reader to \cite{book}. 
A straightforward extension of numerical semigroups to submonoids of $\NN_0^d$, with $d\geq 1$, is presented in \cite{defGNS} where a \emph{generalized numerical semigroup} (or simply GNS for short) is defined as a submonoid $S$ of $\NN_0^d$ whose the complement $\operatorname{H}(S)=\NN_0^d\setminus S$ is finite. Following the usual terminology for numerical semigroups, the elements of $\operatorname{H}(S)$ are called the \emph{gaps} (or \emph{holes}) of $S$ and the number $\operatorname{g}(S)=|\operatorname{H}(S)|$ is called the \emph{genus} of $S$. A first topic regarding this subject concerns with the generators of such monoid, studied in \cite{generator}. 
In particular, an element $\mathbf{x}\in S\setminus \{\mathbf{0}\}$ is a \emph{minimal generator} of $S$ if $\mathbf{x}=\mathbf{y}+\mathbf{z}$, for $\mathbf{y},\mathbf{z}\in S$, implies that either $\mathbf{y}=\mathbf{0}$ or $\mathbf{z}=\mathbf{0}$; it is shown that for every generalized numerical semigroup the set of minimal generators is unique and finite, whose cardinality $\operatorname{e}(S)$ is called the \emph{embedding dimension} of $S$. Furthermore, about the study of such class of submonoids of $\NN_0^d$, in \cite{irreducible} it is considered the class of \emph{irreducible} generalized numerical semigroups, which are generalized numerical semigroups that cannot be written as intersection of other nontrivial generalized numerical semigroups. Generalizations of known concepts and results on irreducible numerical semigroups are provided there. Moreover the notion of Frobenius element with respect to a monomial order is defined together with a class of generalized numerical semigroups in which the Frobenius element is uniquely defined (the so-called \emph{Frobenius GNSs}) and it is shown that irreducible generalized numerical semigroups are maximal among generalized numerical semigroups having a fixed Frobenius element. There exist also some procedures to deal with generalized numerical semigroups, described in \cite{algorithms} and implemented in the GAP \cite{GAP} package \texttt{numericalsgps} \cite{numericalsgps}, useful for instance to produce examples and to check properties. We mention that the aim of generalizing properties and definitions of numerical semigroups to submonoids of $\NN_0^d$ is an actual line of research, considering also different monoids than GNSs, for instance in \cite{onPseudoFrob} the concept of irreducibility is introduced in a more general context.\\
A well-known property on numerical semigroups is that of almost-symmetry, introduced in \cite{defASnumsem} and which extends the notion of irreducibility for numerical semigroups. There are different ways to introduce such a concept; in fact, alternative equivalences to the notion of almost-symmetry in numerical semigroups are explored in \cite{symmetries}, where several properties concerning symmetries are also considered even for their pseudo-Frobenius elements as in Ap\'ery sets. In particular it is proven that almost symmetric numerical semigroups are maximal with respect to the inclusion among the numerical semigroups having fixed the Frobenius number and type. Another problem considered in this subject, addressed in \cite{giventype}, is that of computing the whole set of almost symmetric numerical semigroups with fixed Frobenius number, as well as fixed type. \\
In this work we consider an extension of the notion of almost-symmetry to generalized numerical semigroups. Motivated by the properties in numerical semigroups, we present a broad study of this properties in higher dimensions, extending several descriptions and results. This family extends the class of irreducible generalized numerical semigroups and it is indeed a new family of Frobenius generalized numerical semigroups.\\ This paper is organized as follows. Section 2 collects basic terminologies and results concerning generalized numerical semigroups useful for the rest of the paper, regarding in particular the concept of irreducibility, and it is devoted to introduce an extension for the notion of almost-symmetry to the setting of generalized numerical semigroups and its immediate outcomes. Properties regarding symmetries in almost symmetric generalized numerical semigroups are considered in Section 3; many equivalences to the new concept are proven, in particular characterizations of almost-symmetry regarding the set of pseudo-Frobenius elements and the Ap\'ery set are provided also in this context. In Section 4, it is addressed the issue of computing all almost symmetric generalized numerical semigroups having fixed Frobenius element, in particular we show how it is possible to organize them in a rooted tree. We conclude by providing some possible developments to be considered.

\section{Preliminaries}

Recall that a generalized numerical semigroup (GNS) is a submonoid of $\NN_0^d$ with finite complement in it, in particular numerical semigroups are GNSs with $d=1$. Recall also that if $S\subseteq \NN_0^d$ is a GNS then $\operatorname{H}(S)=\NN_0^d \setminus S$ is called the set of \emph{gaps} (or \emph{holes}) of $S$ and the number $\operatorname{g}(S)=|\operatorname{H}(S)|$ is called the \emph{genus} of $S$. In \cite{irreducible} is presented the generalization of some notions defined for numerical semigroups to the framework of GNSs. This section is devoted to recall such concepts and some of their properties useful for this work, as well as to introduce the class of almost symmetric GNSs. We start by setting up some basic terminologies and their notations.

Given $S\subseteq \NN_0^d$ a GNS, we consider in $\mathbb{Z}^d$ the following partial order:
$$\negy\leq_S\negx \ \ \mbox{if and only if} \ \ \negx-\negy\in S.$$
Setting $S^*$ for the set $S\backslash \{\textbf{0}\}$, we define the set of \emph{pseudo-Frobenius elements} of $S$ as
$$\operatorname{PF}(S)=\{\negh\in \operatorname{H}(S) \ | \ \negh+\negs\in S \ \ \mbox{for all} \ \negs\in S^*\}.$$
Its cardinality $|\operatorname{PF}(S)|$ is called the \emph{type} of $S$, denoted by $\operatorname{t}(S)$. The next result characterizes the pseudo-Frobenius elements of $S$ in terms of $\leq _S$.

\begin{proposition}[\cite{irreducible}, Proposition 1.3] Let $S\subseteq \NN_0^d$ be a GNS. The set $\operatorname{PF}(S)$ is the set of maximal elements of $\operatorname{H}(S)$ with respect to the partial order $\leq_S$.
\label{pseudofrobeniusmaximal}
\end{proposition}

The set of \emph{special gaps} of $S$ is defined to be
$$\operatorname{SG}(S)=\{\negh\in \operatorname{PF}(S) \ | \ 2\negh\in S\}.$$
Its elements satisfy the following property.

\begin{proposition}[\cite{irreducible}, Proposition 2.3] Let $S\subseteq \mathbb{N}^d$ be a GNS and let $\mathbf{x}\in \operatorname{H}(S)$. Then $S\cup \{\mathbf{x}\}$ is a GNS if and only if $\mathbf{x}\in \operatorname{SG}(S).$
\end{proposition}

For an element $\negx\in \NN_0^d$ and $i\in \{1,\ldots,d\}$, the $i$-th component of $\negx$ will be denoted by $x^{(i)}$.  We will use $\negx+\mathbf{1}$ to refer to the element of $\NN_0^d$ whose $i$-th component is $x^{(i)}+1$ for all $i\in \{1,\ldots,d\}$ and the symbol $|\negx|_\times$ to stand for $x^{(1)}\cdots x^{(d)}$, the product of the components of $\negx$. By the natural partial order $\leq$ in $\NN_0^d$ we mean
$$\negy\leq\negx \ \ \mbox{if and only if} \ \ y^{(i)}\leq x^{(i)} \ \mbox{for all} \ i\in \{1,\ldots,d\}.$$
Observe that the latter partial order $\leq$ in $\NN_0^d$ can be naturally extended to a partial order in $\mathbb{Z}^d$, which in special, for $\ \negx,\negy\in \mathbb{Z}^d$, satisfies
$$\negy\leq\negx \ \ \ \mbox{if and only if} \ \ \ \negx-\negy\in \NN_0^d,$$
where $\mathbf{x}-\negy$ refers to the usual difference in $\mathbb{Z}^{d}$.

A GNS $S\subseteq \NN_0^d$ is said to be a \emph{Frobenius GNS} if there exists a unique maximal element $\negf$ in $\operatorname{H}(S)$ with respect to the natural partial order $\leq$ of $\NN_0^d$. We denote by $(S,\negf)$ the Frobenius GNS $S$ with Frobenius element $\negf$.

\subsection{Irreducible GNSs}
In \cite{irreducible}, generalizing the notion of irreducibility given for numerical semigroups, a GNS is defined to be \emph{irreducible} if it cannot be expressed as intersection of two GNSs properly containing it. Some properties of them are introduced there, in particular the main characterizations can be summarized in the following:

\begin{proposition}[\cite{irreducible}]
Let $S\subseteq \NN_0^{d}$ be a GNS. Then $S$ is an irreducible GNS if and only if $|\operatorname{SG}(S)|=1$. Moreover $\operatorname{SG}(S)=\{\negf\}$ if and only if one of the following conditions occurs:
\begin{itemize}
\item $\operatorname{PF}(S)=\{\negf\}$;
\item $\operatorname{PF}(S)=\{\negf,\frac{\negf}{2}\}$.
\end{itemize}
In both cases $S$ is a Frobenius GNS with Frobenius element $\negf$.
\end{proposition}

Using the same terminology for irreducible numerical semigroups, if $\operatorname{PF}(S)=\{\negf\}$ then $S$ is called \emph{symmetric}, and if $\operatorname{PF}(S)=\{\negf,\frac{\negf}{2}\}$ then $S$ is called \emph{pseudo-symmetric}.

\begin{remark}
Observe that a GNS is symmetric if and only if $\operatorname{t}(S)=1$. Moreover every pseudo-symmetric GNS has type $2$ but if $S$ is GNS with $\operatorname{t}(S)=2$ then $S$ could not be pseudo-symmetric. For instance the set $\NN_0^2\setminus \{(0,1),(1,0)\}$ is a GNS with type $2$ but it is not pseudo-symmetric.
\end{remark}

Such classes of GNSs can be easily recognized by their set of gaps and the formulas expressed in the following result:

\begin{theorem}[\cite{irreducible}] Let $S\subseteq \NN_0^d$ be a GNS with genus $g$. Then:
\begin{enumerate}
\item[\rm i)] $\operatorname{PF}(S)=\{\negf\}$ if and only $2g=|\negf+\mathbf{1}|_\times$;
\item[\rm ii)] $\operatorname{PF}(S)=\{\negf,\frac{\negf}{2}\}$ if and only $2g-1=|\negf+\mathbf{1}|_\times$.
\end{enumerate}
\label{irreducible}
\end{theorem}


The next proposition is another useful characterization of irreducible GNSs.

\begin{proposition}[\cite{irreducible}, Proposition 2.6] Let $S\subseteq \NN_0^d$ be a GNS. Then $S$ is irreducible if and only if there exists $\negf \in \operatorname{H}(S)$ such that for every $\negh \in \operatorname{H}(S)$ with $2\negh\neq\negf$ we have that $\negf-\negh \in S$.
\label{differenceFrobenius}
\end{proposition}

In particular it is easy to see that the element $\negf$ mentioned in the previous result is exactly the Frobenius element of the irreducible GNS. Such an element provides a description of irreducible GNSs through a maximality condition.

	\begin{proposition}[\cite{irreducible}, Proposition 4.11] Let $S\subseteq \NN_0^d$ be a GNS. Then $S$ is irreducible with Frobenius element $\negf$ if and only if it is maximal with respect to the inclusion in the set of GNSs not containing $\negf$. In particular if $S$ is irreducible with Frobenius element $\negf$ then it is maximal among Frobenius GNSs having $\negf$ as Frobenius element.
		\label{maxirreducible}
	\end{proposition}

\subsection{Almost symmetric GNSs} In the remainder of this section we introduce a new family of GNSs  by extending ideas formulated for numerical semigroups (see \cite{symmetries}) to the context of GNSs.

Let $S\subseteq \NN_0^d$ be a GNS. For each $\negh\in \NN_0^d$, let us consider the set 
$$\pi(\negh)=\{\negx\in \NN_0^d \ | \ \negx\leq \negh\}.$$
Notice that, according to the definition of $| \, \cdot \, |_\times$, we have $|\pi(\negh)|=|\negh+\mathbf{1}|_\times$. Furthermore, considering the subsets of $\pi(\mathbf{h})$
$$\operatorname{LH}(\negh)=\{\negx\in \operatorname{H}(S) \ | \ \negx\leq \negh\} \qquad \mbox{and} \qquad \operatorname{N}(\negh)=\{\negx\in S \ | \ \negx\leq \negh\},$$
since $\pi(\negh)=\operatorname{LH}(\negh)\cup \operatorname{N}(\negh)$, we thus obtain that $|\operatorname{LH}(\negh)|+|\operatorname{N}(\negh)|=|\negh+\mathbf{1}|_\times$. The following lemma provides an estimate for the cardinalities of $\operatorname{N}(\mathbf{h})$ and $\operatorname{LH}(\mathbf{h})$. 

\begin{lemma}[\cite{irreducible}, Lemma 5.5] Let $S\subseteq \NN^d_0$ be a GNS of genus $g$ and $\mathbf{h}\in \operatorname{H}(S)$. Then $|\operatorname{N}(\mathbf{h})|\leq |\operatorname{LH}(\mathbf{h})|\leq g$.
\label{lemmacardinalities}
\end{lemma}

It follows from the previous result that, if $S$ has genus $g$, then it holds that $$2g\geq |\negh+\mathbf{1}|_\times \ \mbox{for any} \ \negh\in  \operatorname{H}(S).$$
This formula extends the well-know  upper bound of a gap in a numerical semigroup. In general, we may sharp the formula above by employing the type of $S$ in the analysis through the following.

\begin{definition}\label{map} Let $S\subseteq \NN^d_0$ be a GNS and let $\mathbf{h}'\in \operatorname{PF}(S)$. We define the following map:
$$\psi_{\mathbf{h}'} \ : \ \operatorname{N}(\negh') \ \to \ \operatorname{H}(S)\backslash(\operatorname{PF}(S)\backslash \{\negh'\}), \quad \negx\longmapsto \negh'-\negx.$$
\end{definition}
It is easily seen that $\psi_{\mathbf{h}'}$ is a well-defined and injective map. 

\begin{proposition}\label{ASformula} Let $S\subseteq \NN_0^d$ be a GNS with genus $g$ and type $t$. For any $\, \negh\in \operatorname{H}(S)$, it holds that
$$2g+1-t\geq |\negh+\mathbf{1}|_\times.$$
\end{proposition}
\begin{proof} Given $\negh'\in \operatorname{PF}(S)$, since the map $\psi_{\mathbf{h}'}$ is injective, we get
$$|\operatorname{N}(\negh')|\leq |\operatorname{H}(S)\backslash(\operatorname{PF}(S)\backslash \{\negh'\})|=g-(t-1)=g+1-t.$$
Hence, by using Lemma \ref{lemmacardinalities}, we obtain
$$|\negh'+\mathbf{1}|_\times=|\operatorname{LH}(\negh')|+|\operatorname{N}(\negh')|\leq g+(g+1-t)=2g+1-t$$
for any $\negh'\in \operatorname{PF}(S)$. Now, as $\operatorname{PF}(S)$ consists of the maximal elements in $\operatorname{H}(S)$ with respect to $\leq_S$ by Proposition \ref{pseudofrobeniusmaximal}, given an $\negh\in \operatorname{H}(S)$, there exists an $\negh'\in \operatorname{PF}(S)$ such that $\negh\leq_S \negh'$. In particular, we have $\negh\leq \negh'$ and therefore $|\negh+\mathbf{1}|_\times\leq |\negh'+\mathbf{1}|_\times\leq 2g+1-t$, which proves the desired formula for any $\negh\in \operatorname{H}(S)$.
\end{proof}

The upper-bound given in Proposition \ref{ASformula} motivates the following definition.

\begin{definition}\label{ASdef} Let $S$ be a GNS with genus $g$ and type $t$. We say that $S$ is \emph{almost symmetric} if
$$2g+1-t= |\negh+\textbf{1}|_\times$$
for some $\negh\in \operatorname{H}(S)$.
\end{definition}

Observe that Theorem \ref{irreducible} provides the irreducible GNSs as a subclass of almost symmetry GNSs.

\begin{example}\label{ex1} Consider $S=\NN_0^2\setminus \{(1,0),(2,0),(3,0),(4,0),(6,0), (9,0)\}$ the GNS with genus $g=6$ and Frobenius element $(9,0)$. A straightforward verification shows that the set of pseudo-Frobenius elements of $S$ is $\operatorname{PF}(S)=\{(9,0), (6,0), (3,0)\}$, and thus $t=\operatorname{t}(S)=3$. Hence, $S$ is almost symmetric because
$$2g+1-t=12+1-3=|(10,1)|_\times=10\cdot 1=10.$$
\end{example}

\begin{example}\label{ex2} The so-called \emph{ordinary GNSs}, i.e. GNSs of kind $S(\negf)=(\NN_0^d\backslash \pi(\negf))\cup \{\textbf{0}\}$ for some $\negf\in \NN_0^d$, are examples of almost symmetric GNSs. Indeed, for $S(\negf)$ an ordinary GNS, we have $\operatorname{PF}(S(\negf))=\operatorname{H}(S(\negf))$ and then $\operatorname{g}(S(\negf))=\operatorname{t}(S(\negf))$. Hence, it follows that $2\operatorname{g}(S(\negf))+1-\operatorname{t}(S(\negf))=\operatorname{g}(S(\negf))+1=|\negf+\mathbf{1}|_\times$ and $S(\negf)$ is almost symmetric. 
\end{example}



It is natural to wonder how many $\mathbf{h}\in \operatorname{H}(S)$ can satisfy the Definition \ref{ASdef} for $S$ an almost symmetric GNS. The answer is a consequence of the following result:

\begin{proposition} \label{remarkFrobenius}
Every almost symmetric GNS is a Frobenius GNS. In particular if $(S, \negf)$ is an almost symmetric GNS with genus $g$ and type $t$, then $2g+1-t= |\negf+\mathbf{1}|_\times$.
\end{proposition}
\begin{proof}
Let $S\subseteq \NN_0^d$ be an almost symmetric GNS and $\negh \in \operatorname{H}(S)$ such that Definition~\ref{ASdef} is satisfied. It follows from Lemma \ref{lemmacardinalities} and the proof of Proposition~\ref{ASformula} that
$$|\operatorname{LH}(\negh)|\leq g \qquad \mbox{and} \ \qquad |\operatorname{N}(\negh)|\leq g+1-t.$$
Hence, as $|\negh+\mathbf{1}|_\times=|\operatorname{LH}(\negh)|+|\operatorname{N}(\negh)|$, the upper-bound $2g+1-t$ given in Proposition~\ref{ASformula} is reached by $|\negh+\mathbf{1}|_\times$ if and only if equalities hold for the cardinalities $|\operatorname{LH}(\negh)|$ and $|\operatorname{N}(\negh)|$ in the formulas above. Consequently, $\operatorname{H}(S)=\operatorname{LH}(\negh)$, which yields $\negh$ as the unique maximal element in $\operatorname{H}(S)$ with respect to the natural partial order $\leq$. Therefore, for an almost symmetric GNS $S$, there exists exactly one element $\mathbf{h}\in \operatorname{H}(S)$ satisfying Definition \ref{ASdef} and it is the Frobenius element of $S$.
\end{proof}

We mention that it is asked in \cite{irreducible} about the possibility of classifying other classes of Frobenius GNSs than the irreducible ones. The previous proposition moves also in the direction concerning that question.

	\begin{remark} Putting together Propositions \ref{ASformula} and \ref{remarkFrobenius}, we can deduce a  maximality property for almost symmetric GNSs: every almost symmetric GNS is maximal with respect to the inclusion in the set of Frobenius GNSs having fixed the type and the Frobenius element. It occurs  because, according to the aforementioned results, almost symmetric GNSs reach the least genus possible among the Frobenius GNSs with fixed type and Frobenius element. Notice nevertheless that not every GNS with Frobenius element $\negf$ and type $t$ is contained in an almost symmetric GNS with the same Frobenius element and type. For instance, $S=S((2,1))\cup {(0,1)}$ has Frobenius element $(2,1)$ and type $2$ and there are not almost symmetric GNSs with Frobenius element $(2,1)$ and type $2$, since in this case $2g-t+1$ is an odd number, while $|\negf+\mathbf{1}|_\times=6$ is even.
	\end{remark}

\begin{remark}\label{remarkbijection} It is worth noting that Proposition~\ref{remarkFrobenius} provides that a Frobenius GNS $(S,\mathbf{f})$ in $\NN_0^d$ is almost symmetric if and only if the map $\psi_{\mathbf{f}}$ is a bijection. Note further that this equivalence yields a generalization of Proposition \ref{differenceFrobenius} since, in view of Definition \ref{map}, it means that $S$ is almost symmetric if and only if for every $\mathbf{h}\in \operatorname{H}(S)\setminus (\operatorname{PF}(S)\setminus\{\mathbf{f}\})$ we have that $\mathbf{f}-\mathbf{h}\in S$.
\end{remark}

\section{Symmetries in almost symmetric GNSs}

In this section we investigate properties regarding symmetries of the pseudo-Frobenius elements and elements in Apéry sets occurring in almost symmetric GNSs. The study of such properties follows the spirit of \cite{symmetries} to extend many results given for numerical semigroups to the GNS setting, providing some equivalences to the notion of almost-symmetry.

We begin recalling that a total order $\prec$ in $\NN_0^{d}$ is called a \emph{monomial order} if it satisfies:
\begin{itemize}
\item[1)]if $\mathbf{v},\mathbf{w}\in \mathbb{N}_0^{d}$ with $\mathbf{v}\prec\mathbf{w}$ then $\mathbf{v}+\mathbf{u}\prec\mathbf{w}+\mathbf{u}$ for every $\mathbf{u}\in \mathbb{N}_0^{d}$; and
\item[2)]if $\mathbf{v}\in \mathbb{N}_0^{d}$ and $\mathbf{v}\neq\mathbf{0}$ then $\mathbf{0}\prec\mathbf{v}$.
\end{itemize}
Recall further that every monomial order in $\NN_0^d$ extends the natural partial order in $\NN_0^d$ (see also \cite[Proposition 4.4]{irreducible}). The next technical lemma is about certain finite sets in $\NN_0^d$ ordered by a monomial order.

\begin{lemma}
Let $\prec$ be a monomial order in $\NN_0^{d}$ and let $\{\negx_{1}\prec \negx_{2}\prec \cdots \prec \negx_{t-1}\}$ be a subset of $\NN_0^{d}$. Suppose that there exists $\negf\in \NN_0^d$ such that for all $i\in \{1,\ldots,t-1\}$ there exists $j\in \{1,\ldots,t-1\}$ with $\negx_i+\negx_j=\negf$. Then $\negx_i+\negx_{t-i}=\negf$ for all $i\in \{1,\ldots,t-1\}$.
\label{ordered}
\end{lemma}
\begin{proof}
Suppose that, contrary to our claim, there is $i\in \{1,\ldots,t-1\}$ such that $\negx_i+\negx_{t-i}\neq \negf$. Hence, we can choose $i=\min \{\,l \ \mid \ 1\leq l\leq t-1 \ \mbox{with} \ \negx_l+\negx_{t-l}\neq \negf\}$. By assumption, for such $i$ there exists $j\in \{1,\ldots,t-1\}$ such that $\negx_i+\negx_j=\negf$ with $j\neq t-i$. If $j>t-i$ then $j=t-(i-k)$ with $0<k\leq t-1$. Moreover $i-k<i$, and so $\negx_{i-k}+\negx_{t-(i-k)}=\negf =\negx_i+\negx_j$, that is, $\negx_i=\negx_{i-k}$, which gives a contradiction. Otherwise, if $j<t-i$, consider $k\in \{1,\ldots,t-1\}$ such that $\negx_{t-i}+\negx_{k}=\negf$. Since $k\neq i$, we have the following two cases:
\begin{itemize}
\item If $k<i$ then $\negx_{k}+\negx_{t-k}=\negf=\negx_{t-i}+\negx_{k}$, in particular $\negx_{t-k}=\negx_{t-i}$, which yields $k=i$, a contradiction.
\item If $k> i$ then $\negf=\negx_i+\negx_j\prec \negx_i+\negx_{t-i}\prec \negx_k+\negx_{t-i}=\negf$, a contradiction.
\end{itemize}
Since all cases give a contradiction, it holds that $\negx_i+\negx_{t-i}= \negf$ for all $i\in \{1,\ldots,t-1\}$.
\end{proof}


\subsection{Pseudo-Frobenius elements} We show in the following that the set of pseudo-Frobenius elements plays an important role in the determination of the property of almost-symmetry in GNSs.


\begin{definition} Let $(S,\mathbf{f})$ be a Frobenius GNS in $\NN_0^{d}$. We define the related sets to $S$:
\begin{enumerate}[\rm (a)]
\item $\operatorname{Z}_{S}=\{\mathbf{x}\in \mathbb{Z}^d \mid \mathbf{x}\leq \mathbf{f}\}\cup \mathbb{N}_0^d$, observe that $\operatorname{H}(S)\subset \operatorname{Z}_{S}$;
\item $\operatorname{L}(S)=\{\negh\in \operatorname{H}(S) \mid \ \exists \ \negh'\in \operatorname{H}(S), \ \negh+\negh'=\negf\}$.
\end{enumerate}
\end{definition}

The next result establishes some equivalent conditions for a GNS to be almost symmetric in terms of their pseudo-Frobenius elements.

\begin{proposition} Let $(S,\negf)$ be a Frobenius GNS in $\NN_0^d$ with type $t$. Then the following statements are equivalent:
\begin{enumerate}[\rm i)]
\item  $S$ is almost symmetric;
\item $\operatorname{L}(S)\subseteq \operatorname{PF}(S)$;
\item $\operatorname{PF}(S)=\operatorname{L}(S)\cup \{\negf\}$;
\item if $\mathbf{h}\in \operatorname{Z}_{S}\setminus S$ then either $\mathbf{f}-\mathbf{h}\in S$ or $\mathbf{f}-\mathbf{h}\in \operatorname{PF}(S)$;
\item if $\negh\in \operatorname{H}(S)$ then either $\negf-\negh\in S$ or $\negf-\negh\in \operatorname{PF}(S)$;
\item for every monomial order $\prec$ in $\NN_0^{d}$, if  $$\operatorname{PF}(S)\setminus \{\mathbf{f}\}=\{\mathbf{x}_{1}\prec \mathbf{x}_{2}\prec \cdots \prec \mathbf{x}_{t-1}\},$$ then 
$$\mathbf{x}_{i}+\mathbf{x}_{t-i}=\mathbf{f} \quad \mbox{for} \ \ i=1,\ldots,t-1.$$
\item there exists a monomial order $\prec$ in $\NN_0^{d}$ such that if $$\operatorname{PF}(S)\setminus \{\mathbf{f}\}=\{\mathbf{x}_{1}\prec \mathbf{x}_{2}\prec \cdots \prec \mathbf{x}_{t-1}\},$$ then 
$$\mathbf{x}_{i}+\mathbf{x}_{t-i}=\mathbf{f} \quad \mbox{for} \ \ i=1,\ldots,t-1.$$
\end{enumerate}
\label{bigprop}
\end{proposition}
\begin{proof}
i) $\Rightarrow$ ii): Let $\negh\in \operatorname{L}(S)\backslash \operatorname{PF}(S)$. Since the map $\psi_\negf$ (cf. Definition \ref{map}) is a bijection by Remark \ref{remarkbijection}, we obtain that $\negf-\negh\in S$, contradicting $\negh\in \operatorname{L}(S)$.
\\ ii) $\Rightarrow$ iii): As $\operatorname{L}(S)\cup \{\negf\}\subseteq \operatorname{PF}(S)$, let $\negh\in \operatorname{PF}(S)\setminus (\operatorname{L}(S)\cup \{\negf\})$. Hence, $\negf-\negh\in S^*$ and consequently, $\negf=\negh+(\negf-\negh)\in S$, which is a contradiction.
\\ iii) $\Rightarrow$ iv): For $\mathbf{h}\in \operatorname{Z}_{S}\setminus \, S$, as $\operatorname{Z}_{S}\setminus \, S=\operatorname{H}(S)\cup (\{\negx\in \mathbb{Z}^d\mid \negx\leq \negf\}\setminus  \mathbb{N}_0^d)$, in every case it holds that $\negh\leq \negf$ and we thus have $\mathbf{f}-\mathbf{h}\in \NN_0^{d}$. Suppose that $\mathbf{f}-\mathbf{h}\notin \operatorname{PF}(S)$, otherwise we are done. We obtain by iii) that $\mathbf{f}-\mathbf{h}\notin \operatorname{L}(S)$ and thus we have two possibilities. First, if $\mathbf{f}-\mathbf{h}\notin \operatorname{H}(S)$ then $\mathbf{f}-\mathbf{h}\in S$ (since $\mathbf{f}-\mathbf{h}\in \NN_0^{d}$) and we are done. Second, if $\mathbf{f}-\mathbf{h}\in \operatorname{H}(S)$ then $\mathbf{f}\geq \mathbf{f}-\mathbf{h}$ and in particular $\mathbf{h}=\mathbf{f}-(\mathbf{f}-\mathbf{h})\in \NN_0^{d}$. Since $\mathbf{h}\notin S$, we get $\mathbf{h}\in \operatorname{H}(S)$. So we obtain $\mathbf{f}-\mathbf{h}\in \operatorname{H}(S)$ and $\mathbf{f}-(\mathbf{f}-\mathbf{h})=\mathbf{h}\in \operatorname{H}(S)$, that is, $\mathbf{f}-\mathbf{h}\in \operatorname{L}(S)$, contrary to our assumption.
\\ iv) $\Rightarrow$ v): This is an immediate consequence since $\operatorname{H}(S)\subset \operatorname{Z}_{S}\setminus \, S$.
\\ v) $\Rightarrow$ vi): Let $\negx_{i}\in \operatorname{PF}(S)\setminus \{\negf\}$. Notice that $\negf-\negx_i \not\in S$ since otherwise, as $\negx_i\in \operatorname{PF}(S)$, we would have $\negf\in S$. Hence, it follows from v) that $\negf-\negx_i\in \operatorname{PF}(S)\setminus \{\negf\}$, and consequently $\negx_i+\negx_j=\negf$ for some $j\in \{1,\ldots,t-1\}$. The claim thus follows by Lemma~\ref{ordered}. 
\\ vi) $\Rightarrow$ vii): This implication is trivial.
\\ vii) $\Rightarrow$ i): By Remark \ref{remarkbijection}, it suffices to prove that the map $\psi_\negf$ (cf. Definition \ref{map}) is bijective, that is, $\mathbf{h}\in  \operatorname{H}(S)\backslash \operatorname{PF}(S)$ implies $\mathbf{f}-\mathbf{h}\in S$. For $\mathbf{h}\in  \operatorname{H}(S)\backslash \operatorname{PF}(S)$, it follows from Proposition \ref{pseudofrobeniusmaximal} that $\mathbf{h}\leq_S \mathbf{g}$ for some $\mathbf{g}\in \operatorname{PF}(S)$. If $\mathbf{g}=\mathbf{f}$ then we are done because it gives $\mathbf{f}-\mathbf{h}\in S$. Otherwise, we have $\mathbf{h}\leq_S\mathbf{x}_i$ for some $i\in \{1,\ldots,t-1\}$, and thus $(\mathbf{f}-\mathbf{x}_{t-i})-\mathbf{h}=\mathbf{x}_i-\mathbf{h}\in S^*$ by vii). Hence $\mathbf{f}-\mathbf{h}=\mathbf{x}_{t-i}+((\mathbf{f}-\mathbf{x}_{t-i})-\mathbf{h})\in S$.
\end{proof}

As a consequence, we obtain the following description for pseudo-symmetric GNSs.
\begin{corollary}
Let $S\subseteq \NN_0^d$ be a GNS. Then $S$ is pseudo-symmetric if and only if $S$ is almost symmetric and $\operatorname{t}(S)=2$.
\end{corollary}
\begin{proof}
If $S$ is pseudo-symmetric it is easy to see that $S$ has type 2 and it is almost symmetric. Conversely, assume that $S$ is almost symmetric with Frobenius element $\negf$ and type 2. Then $\operatorname{PF}(S)\setminus \{\negf\}=\{\negx\}$ and by vi) of Proposition~\ref{bigprop} we have $2\negx=\negf$. So $\operatorname{PF}(S)=\{\negf,\negf/2\}$, that is, $S$ is pseudo-symmetric.
\end{proof}

\begin{remark}
Let $S$ be a submonoid of $\NN_0^d$. We can define a \emph{relative ideal} of $S$ a subset $I$ of $\mathbb{Z}^d$ such that $I+S\subseteq I$ and $s+I\subseteq S$ for some $s\in S$, generalizing well-known concepts as expressed in \cite{Barucci} and \cite{symmetries}. For an almost symmetric numerical semigroup $S$ (that is, when $d=1$), $\operatorname{Z}_S=\mathbb{Z}$ and the condition number iv) in Proposition \ref{bigprop} can be rewritten as: if $\mathbf{h}\in \mathbb{Z}\setminus S$ then either $\mathbf{f}-\mathbf{h}\in S$ or $\mathbf{f}-\mathbf{h}\in \operatorname{PF}(S)$, where $\mathbf{f}$ is the Frobenius number. The natural generalization of this condition to a Frobenius GNS $(S, \mathbf{f})$ in $\NN_0^d$ with $d\geq 1$ thus becomes: if $\mathbf{h}\in \mathbb{Z}^{d}\setminus S$ then either $\mathbf{f}-\mathbf{h}\in S$ or $\mathbf{f}-\mathbf{h}\in \operatorname{PF}(S)$.
However, this generalized formulation does not hold for almost symmetric GNS $S\subseteq \NN_0^d$ when $d>1$ as we can see in the next example.
This behavior can be related to the fact that the set $\operatorname{K}_{S}=\{\mathbf{f}-\mathbf{h}\mid \mathbf{h}\in \mathbb{Z}^{d}\setminus S\}$ is not a relative ideal, in the sense that it is not true that $\mathbf{s}+\operatorname{K}_{S}\subseteq S$ for some $\mathbf{s}\in S$. Notice nevertheless that it works if we consider $\overline{\operatorname{K}}_{S}=\{\mathbf{f}-\mathbf{h}\mid \mathbf{h}\in \operatorname{Z}_{S}\setminus S\}$ instead of $\operatorname{K}_{S}$.
\end{remark}

\begin{example}
Let $S=\NN_0^{2}\setminus \{(0,1),(0,2)\}$. Since $S$ is pseudo-symmetric, it is almost symmetric and has Frobenius element $\mathbf{f}=(0,2)$. For all $n\in \NN_0$, consider that  $(-n,3)\notin \operatorname{Z}_{S}$, that is $\mathbf{f}-(-n,3)=(n,-1)\notin \NN_0^{2}$, so it does not belong to $S$ and $\operatorname{PF}(S)$. Moreover $\operatorname{K}_S$ is not a relative ideal of $S$. If it is true that $\mathbf{s}+\operatorname{K}_{S}\subseteq S$ for some $\mathbf{s}=(s_{1},s_{2})\in S$, taking $\mathbf{t}=\mathbf{f}-\mathbf{h}\in \operatorname{K}_{S}$ with $\mathbf{h}=(h_{1},h_{2})$ such that $h_{1}<0$ and $h_{2}>s_{2}+2$, we have $\mathbf{s}+\mathbf{t}=(s_{1}-h_{1},s_{2}+2-h_{2})\notin \NN_0^{2}$, and then it does not belong to $S$. An example of element belonging to $\operatorname{Z}_S$ is $(-n,1)$ with $n\in \NN_0^d$.
\end{example}

\subsection{Apéry and related sets}

Ap\'ery sets are involved with the topic of almost-symmetry in numerical semigroups through the characterizations given in \cite{symmetries}. In what follows, after recalling some useful tool, we introduce a special subset of Ap\'ery sets that preserves some nice properties of Ap\'ery sets from the numerical semigroup case and allows to yield a generalization of that characterizations in the context of GNSs. We start by bringing the definition of Ap\'ery set for submonoids of $\NN_0^d$.

\begin{definition} Let $S\subseteq \NN_0^{d}$ be a monoid and $\negn\in S$. The \emph{Ap\'ery set} of $S$ with respect to $\negn$ is the set
$$ \operatorname{Ap}(S,\negn)=\{\mathbf{s}\in S  \mid  \mathbf{s}-\negn\notin S\},$$
where $\mathbf{s}-\negn$ stands for the usual difference in $\mathbb{Z}^{d}$. \end{definition}

In the GNS setting, pseudo-Frobenius elements are related to some elements of an Ap\'ery set through the following result:

\begin{proposition}[\cite{irreducible}, Proposition 1.4] Let $S\subseteq \NN_0^{d}$ be a GNS and $\negn\in S^*$. Then
$$\operatorname{PF}(S)=\{\mathbf{w}-\negn\mid \mathbf{w}\in \mathrm{Maximals}_{\leq_{S}}\operatorname{Ap}(S,\negn)\}.$$
\label{pseudoApery}
 \end{proposition}

\begin{remark} Let $S\subseteq \NN_0^{d}$ be a monoid and $\mathbf{s}_{1},\mathbf{s}_{2}\in S$ with $\mathbf{s}_{1}+\mathbf{s}_{2} \in \operatorname{Ap}(S,\negn)$. Then $\mathbf{s}_{1},\mathbf{s}_{2} \in \operatorname{Ap}(S,\negn)$. In fact if $\mathbf{s}_{1}-\negn\in S$, then $\mathbf{s}_{1}+\mathbf{s}_{2}-\negn \in S$, that is a contradiction. \label{sumAp}\end{remark}

The main difference between the Ap\'ery set of a submonoid of $\NN_0$ and a submonoid of $\NN_0^d$ with $d>1$ is that in second case the Ap\'ery set can contain infinitely many elements. In particular, it is easy to see that for every GNS in $\NN_0^d$ with $d>1$ (and $S\neq \NN_0^d$) then $\operatorname{Ap}(S,\negn)$ contains infinitely many elements for every $\negn\in S^*$. 

\begin{example} 
Let $S=\mathbb{N}^{2}\setminus\{(1,0),(1,1)\}$. The Ap\'ery set of $S$ with respect to $(0,1)\in S$ is $\operatorname{Ap}(S,(0,1))=\{(0,0),(1,2),(n,0) |\ n\geq 2\}$.
\end{example}

For our arguments, we focus the attention on a particular finite subset of the Ap\'ery set, introduced for the first time in \cite{phdthesis}.

\begin{definition} Let $S\subseteq \NN_0^d$ be a GNS and $\negn\in S$. We define the \emph{reduced Ap\'ery set} of $S$ with respect to $\negn$ as
$$\operatorname{C}(S,\negn)=\left\lbrace \mathbf{s}\in \operatorname{Ap}(S,\negn)\mid \mathbf{s}\leq \mathbf{h}+\negn \ \mbox{for some}\ \mathbf{h}\in \operatorname{H}(S)\right\rbrace.$$
\end{definition}

Observe that for numerical semigroups $\operatorname{C}(S,\negn)=\operatorname{Ap}(S,\negn)$. Some useful relations from $\operatorname{C}(S,\negn)$ and $\operatorname{Ap}(S,\negn)$ are given in the following result.

\begin{proposition}
Let $S\subseteq \NN_0^{d}$ be a GNS and $\negn\in S$. The following assertions are verified:
\begin{itemize}
\item[1)] $\mathrm{Maximals}_{\leq}\operatorname{Ap}(S,\negn)=\mathrm{Maximals}_{\leq}\operatorname{C}(S,\negn)=\{\mathbf{h}+\negn\mid \mathbf{h}\in \mathrm{Maximals}_{\leq}\operatorname{H}(S)\}$;
\item[2)] $\mathrm{Maximals}_{\leq_{S}}\operatorname{Ap}(S,\negn)\subseteq\mathrm{Maximals}_{\leq_{S}}\operatorname{C}(S,\negn)$. 
\end{itemize}
\label{maximals}
\end{proposition}
\begin{proof} Let $\nege_1,\ldots,\nege_d$ be the standard basis vectors of $\NN_0^d$.\\
1) Let $\mathbf{a}\in S$ be maximal in $\operatorname{Ap}(S,\mathbf{n})$ with respect to $\leq$. Then $\mathbf{a}-\mathbf{n}\notin S$. If $\mathbf{a}-\mathbf{n}\notin \NN_0^d$ then there exists $i\in \{1,2,\ldots,d\}$ such that it is possible to write $\mathbf{a}=\mathbf{a}'+a^{(i)}\mathbf{e}_{i}$ and $\mathbf{n}=\mathbf{n}'+n^{(i)}\mathbf{e}_{i}$, where $\mathbf{a}',\mathbf{n}'\in \mathbb{N}_0^{d}$ whose $i$-th component is zero, and $a^{(i)}<n^{(i)}$. Since $\operatorname{H}(S)$ is finite, the components of all element in $\operatorname{H}(S)$ are bounded, therefore there exists $t \in \NN_0$ such that $\negx=\mathbf{a}+t\mathbf{e}_j\in S$, with $j\neq i$. In particular $\mathbf{a}\leq \negx$ and $\negx-\negn \notin S$, which contradicts the maximality of $\mathbf{a}$. So $\mathbf{a}-\mathbf{n}\in \NN_0^d$, that is, there exists $\mathbf{h}$ maximal in $\operatorname{H}(S)$ with respect to $\leq$, such that $\mathbf{a}-\mathbf{n}\leq \mathbf{h}$. Then $\mathbf{a}\leq \mathbf{h}+\mathbf{n}$ and $\mathbf{h}+\mathbf{n}\in \operatorname{Ap}(S,\mathbf{n})$. By maximality of $\mathbf{a}$ it is verified that $\mathbf{a}=\mathbf{h}+\mathbf{n}$. It follows by definition that $\operatorname{C}(S,\mathbf{n})$ has the same maximal elements.\\
2) Let $\mathbf{a}$ be maximal in $\operatorname{Ap}(S,\mathbf{n})$ with respect to $\leq_{S}$. Then $\negh=\nega -\negn\in \operatorname{PF}(S)$, by Proposition~\ref{pseudoApery}. In particular $\nega=\negh+\negn$, that is $\nega \in \operatorname{C}(S,\negn)$. If there exists $\mathbf{s}\in \operatorname{C}(S,\mathbf{n})$ such that $\mathbf{a}\leq_{S}\mathbf{s}$ we obtain a contradiction with the maximality of $\mathbf{a}$ in $\operatorname{Ap}(S,\mathbf{n})$, since $\mathbf{s}\in \operatorname{Ap}(S,\mathbf{n})$. 
\end{proof}

In general it is verified that
$\mathrm{Maximals}_{\leq_{S}}\operatorname{Ap}(S,\mathbf{n})\subsetneq\mathrm{Maximals}_{\leq_{S}}\operatorname{C}(S,\mathbf{n})$, as we can see in the following example:\ 

\begin{example}
 
Let $S=\mathbb{N}_0^{2}\setminus \{(0,1),(0,3),(1,0),(1,1),(1,3),(2,1),(2,0),(3,0)\}$. Let us compute the sets $\operatorname{Ap}(S,\negn)$ and $\operatorname{C}(S,\negn)$, for $\negn=(4,0)$.

\vspace{3pt}
$\operatorname{Ap}(S,(4,0))=\{(0,0),(0,2),(0,n),(1,2),(1,n),(2,2),(2,3),(2,n),(3,1),(3,2),(3,3),(3,n),\newline(4,1),(4,3),(5,0),(5,1),(5,3),(6,1),(6,0),(7,0)\mid n\geq 4\}$

\vspace{3pt}
$\operatorname{C}(S,(4,0))=\{(0,0),(0,2),(1,2),(2,2),(2,3),(3,1),(3,2),(3,3),(4,1),(4,3),(5,0),(5,1),\newline (5,3),(6,1),(6,0),(7,0)\}$

\vspace{3pt}
We can see that:
\begin{itemize}
\item Maximals$_{\leq}\operatorname{Ap}(S,\negn)=\{(7,0),(6,1),(5,3)\}$
\item Maximals$_{\leq}\operatorname{C}(S,\negn)=\{(7,0),(6,1),(5,3)\}$ 
\item Maximals$_{\leq_{S}}\operatorname{Ap}(S,\negn)=\{(4,3),(5,0),(5,3),(6,1),(6,0),(7,0)\}$
\item Maximals$_{\leq_{S}}\operatorname{C}(S,\negn)=\{(4,3),(5,0),(5,3),(6,1),(6,0),(7,0),(3,3),(3,2),(2,3)\}$

\end{itemize} 

\noindent Observe that $(3,3)$ is not maximal in $\operatorname{Ap}(S,(4,0))$ with respect to $\leq_{S}$, since $(3,5)-(3,3)=(0,2)\in S$ and $(3,5) \in \operatorname{Ap}(S,(4,0))$. But $(3,3)\in\mathrm{Maximals}_{\leq_{S}}\operatorname{C}(S,\mathbf{n})$.

\label{esempioC}
\end{example}


\begin{theorem}
Let $S\subseteq \NN_0^{d}$ be a GNS, $\negn\in S$ and let $\prec$ be a monomial order in $\NN_0^{d}$. Then $S$ is almost symmetric with type $t$ if and only if  it is possible to arrange the set $\operatorname{C}(S,\negn)$ in the following way
$$\operatorname{C}(S,\negn)=\{\mathbf{a}_{0}=\mathbf{0}\prec\mathbf{a}_{1}\prec \cdots \prec \mathbf{a}_{m}\}\cup \{\mathbf{b}_{1}\prec \cdots \prec\mathbf{b}_{t-1} \},$$ 
with the following conditions:
\begin{enumerate}[\rm 1.]
\item $\negb_{t-1}\prec \nega_m$.
\item $\negb_i-\negn \in \mathbb{N}_0^d$ for every $i\in \{1,\ldots,t-1\}$.
\item $\negb_i \nleq_{S} \negb_j$ for every $i\neq j$.
\item $\mathbf{a}_{i}+\mathbf{a}_{m-i}=\mathbf{a}_{m}$ for all $i \in \{0,1,\ldots,m\}$.
\item $\mathbf{b}_{j}+\mathbf{b}_{t-j}=\mathbf{a}_{m}+\mathbf{n}$ for all $j \in \{1,\ldots,t-1\}$.
\end{enumerate}\label{AsymAp} \end{theorem}

\begin{proof}
$(\Rightarrow)$ Suppose that $S$ is almost symmetric with type $t$, in particular $S$ has a unique Frobenius element $\negf$. By Proposition~\ref{maximals}, $\mathrm{Maximals}_{\leq}\operatorname{C}(S,\negn)=\{\negf+\negn\}$, and let $\nega_{m}=\negf+\negn$. Observe that if $\negx \in \operatorname{PF}(S)$ then it is possible to express $\negx=\negb-\negn$ with $\negb \in S$. So, let $\operatorname{PF}(S)=\{\mathbf{b}_{i}-\mathbf{n}, \mathbf{a}_{m}-\negn\mid i=1,\ldots,t-1\}$, with $\mathbf{b}_{1}\prec \mathbf{b}_{2}\prec \ldots \prec \mathbf{b}_{t-1}$. In particular $\mathbf{b}_{i}\in \operatorname{C}(S,\negn)$ for every $i\in \{1,\ldots,t-1\}$ and observe that $\negb_i \nleq_{S} \negb_j$ for every $i\neq j$ and $\mathbf{b}_{t-1}\prec \nega_{m}$. Consider the two disjoint set $\{\mathbf{a}_{0}=\mathbf{0}\prec\mathbf{a}_{1}\prec \cdots \prec \mathbf{a}_{m}\}$ and $\{\mathbf{b}_{1}\prec \cdots \prec \mathbf{b}_{t-1} \}$, whose union is $\operatorname{C}(S,\negn)$. By Lemma~\ref{ordered} it suffices to prove that 
\begin{enumerate}
\item for every $\nega_i$ there exists $\nega_j$ such that $\nega_i+\nega_j=\nega_m$; and
\item for every $\negb_i$ there exists $\negb_j$ such that $\negb_i+\negb_j=\nega_m+\negn$.
\end{enumerate}
\noindent (1) Let $i\in \{1,\ldots,m-1\}$ (for $i=0$ and $i=m$ it is trivial) and consider the element $\nega_i$. We have that $\nega_i-\negn \notin S$ and since $\nega_i \in \operatorname{C}(S,\negn)$ then $\negf-(\nega_i-\negn)=\negf+\negn-\nega_i\in \NN_0^d$, that is $\nega_i-\negn \in Z_{S}$. Let $\negx=\nega_m-\nega_i=\negf-(\nega_i-\negn)$, since $S$ is almost symmetric then $\negx\in S$ or $\negx \in \operatorname{PF}(S)$, by iv) of Proposition~\ref{bigprop}. If $\negx \in \operatorname{PF}(S)$ then, by vi) of Proposition~\ref{bigprop}, $\nega_i-\negn=\negf-\negx \in \operatorname{PF}(S)$, that is a contradiction with the definition of $\nega_i$. So $\negx \in S$, in particular $\nega_m=\nega_i+\negx$ and it follows, by Remark~\ref{sumAp}, that $\negx \in \operatorname{Ap}(S,\negn)$. Moreover $\negx \leq \negf +\negn$, so $\negx \in \operatorname{C}(S,\negn)$. If $\negx = \negb_j$ for some $j\in \{1,\ldots,t-1\}$ then $\negf=\nega_i+(\negb_j-\negn)$, that is a contradiction since $\nega_i\in S$ and $\negb_j-\negn \in \operatorname{PF}(S)$. We can conclude that $\negx=\nega_j$ for some $j\in \{1,\ldots,m-1\}$ and $\nega_m=\nega_i+\nega_j$.\\ 
(2) Let $i\in \{1,\ldots,t-1\}$ and consider the element $\negb_i$. Then $\negb_i-\negn \in \operatorname{PF}(S)$ and by vi) of Proposition~\ref{bigprop} we have $(\negb_i-\negn)+\negx=\negf$ for some $\negx \in \operatorname{PF}(S)\setminus \{\negf\}$. In particular $\negx=\negb_j-\negn$ for some $j\in \{1,\ldots,t-1\}$, that is $\negb_i+\negb_j=\negf+2\negn=\nega_m+\negn$.
\\$(\Leftarrow)$ We want to prove that $S$ is almost symmetric of type $t$ by showing that $\nega_m-\negn$ is the Frobenius element of $S$ and $\operatorname{PF}(S)\setminus \{\nega_m-\negn\}=\{\negb_1-\negn \prec \negb_2-\negn\prec \cdots \prec \negb_{t-1}-\negn\}$, since in this case $(\negb_i-\negn)+(\negb_{t-i}-\negn)=\nega_m+\negn-2\negn=\nega_m-\negn$ for all $i=1,\ldots,t-1$, that is, $S$ is almost symmetric by vii) of Proposition~\ref{bigprop}.\\
By the fact that $\mathbf{a}_{i}+\mathbf{a}_{m-i}=\mathbf{a}_{m}$ for all $i \in \{0,1,\ldots,m\}$ and $\mathbf{b}_{j}+\mathbf{b}_{t-j}=\mathbf{a}_{m}+\mathbf{n}$ for all $i \in \{1,\ldots,t-1\}$, we can easily obtain that $\mathrm{Maximals}_{\leq}\operatorname{C}(S,\negn)=\{\nega_m\}$, in particular $S$ is Frobenius with Frobenius element $\negf=\nega_m-\negn$. Furthermore $\nega_i\leq_{S} \nega_m$ for all $i \in \{0,1,\ldots,m-1\}$, that is, $\nega_i -\negn \notin \operatorname{PF}(S)$ for every $i\in \{0,1,\ldots,m-1\}$ from Proposition~\ref{pseudoApery}. Therefore $\operatorname{PF}(S)\setminus \{\nega_m-\negn\}\subseteq \{\negb_1-\negn \prec \negb_2-\negn\prec \cdots \prec \negb_{t-1}-\negn\}$ from Proposition~\ref{pseudoApery} and 2) of Proposition~\ref{maximals}. Suppose that $\negb_i-\negn \notin \operatorname{PF}(S)$ for some $i \in \{0,1,\ldots,t-1\}$. Since $\negb_i-\negn \in \mathbb{N}_0^d$ then $\negb_i-\negn \in \operatorname{H}(S)$ and there exists $\negs \in S^*$ such that $\negb_i-\negn+\negs \in \operatorname{PF}(S)$. If $\negb_i-\negn+\negs =\negb_{j}-\negn$ for some $j\neq i$ we easily obtain that $\negb_i\leq_{S}\negb_j$, that is a contradiction. So the only possibility is $\negb_i-\negn+\negs=\negf$. In such a case we have $\negb_i+\negs+\negn=\nega_m+\negn$, and we obtain $\negs+\negn=\negb_{t-i}$, in particular $\negs=\negb_{t-i}-\negn \notin S$, that is a contradiction for the definition of $\negs$. So $\negb_i-\negn \in \operatorname{PF}(S)$ for every $i\in\{1,\ldots,t-1\}$.
\end{proof}

Observing that for numerical semigroups $\operatorname{C}(S,\negn)=\operatorname{Ap}(S,\negn)$, the previous result provides a generalization to the GNSs framework of the description of almost symmetric numerical semigroups in respect of Ap\'ery sets (cf. \cite[Theorem 2.4]{symmetries}).

\begin{remark}
If $S\subseteq \NN_0^d$ is an almost symmetric GNS, $\negn\in S$ and $\operatorname{C}(S,\negn)=\{\mathbf{a}_{0}=\mathbf{0}\prec\mathbf{a}_{1}\prec \cdots \prec \mathbf{a}_{m}\}\cup \{\mathbf{b}_{1}\prec \cdots \prec\mathbf{b}_{t-1} \}$ with the hypotheses of Theorem~\ref{AsymAp}, then from the proof of the theorem we have that  $\operatorname{PF}(S)=\{\mathbf{b}_{i}-\mathbf{n}, \mathbf{a}_{m}-\negn\mid i=1,\ldots,t-1\}$.
\label{remarkAp}
\end{remark}

\begin{corollary}
Let $S\subseteq \NN_0^d$ be an almost symmetric GNS and $\negn \in S$. Then $$\mathrm{Maximals}_{\leq_{S}}\operatorname{Ap}(S,\negn)=\mathrm{Maximals}_{\leq_{S}}\operatorname{C}(S,\negn).$$
\end{corollary}
\begin{proof}
Consider $\operatorname{C}(S,\negn)=\{\mathbf{a}_{0}=\mathbf{0}\prec\mathbf{a}_{1}\prec \cdots \prec \mathbf{a}_{m}\}\cup \{\mathbf{b}_{1}\prec \cdots \prec\mathbf{b}_{t-1} \}$ with the hypotheses of Theorem~\ref{AsymAp}. Since $\mathbf{a}_{i}+\mathbf{a}_{m-i}=\mathbf{a}_{m}$ for all $i \in \{0,1,\ldots,m\}$ then $\nega_i\leq_{S} \nega_m$ for all $i\in \{0,1,\ldots,m-1\}$. Moreover  $\negb_{i}\nleq_{S}\nega_m$ for every $i\in \{1,\ldots,t-1\}$, since $\negb_i+(\negb_{t-i}-\negn)=\nega_m$ and $\negb_{t-i}-\negn\in \operatorname{H}(S)$. So $\mathrm{Maximals}_{\leq_{S}}\operatorname{C}(S,\negn)=\{\mathbf{b}_{1}, \ldots,\mathbf{b}_{t-1},\nega_m \}$. Let $\nega_m$ and $\negb_i$ for $i\in\{1,\ldots,t-1\}$, then from Remark~\ref{remarkAp} we have that $\nega_m-\negn,\negb_i-\negn \in \operatorname{PF}(S)$ and from Proposition~\ref{pseudoApery} we have $\nega_m,\negb_i\in \mathrm{Maximals}_{\leq_{S}}\operatorname{Ap}(S,\negn) $.
\end{proof}

Hence, for almost symmetric GNSs is verified the equality in the claim 2) of Proposition~\ref{maximals}. It is an open question to know if there exist other classes of generalized numerical semigroups with this property.

\begin{example}
Let $S=\NN_0^2\setminus \{(0,1),(0,2),(1,0),(1,1),(1,2),(2,1), (2,2),(3,2)\}$ be the semigroup generated by the set $\{(0,3), (0,5), (0,4), (3,0),(2,0),(1,3),(1,4), (3,1), (1,5),\newline (4,1), (4,2), (5, 2)\}$. $S$ has Frobenius element $\negf=(3,2)$ and $\operatorname{g}(S)=8$.\\ 
We have $\operatorname{PF}(S)=\{(1,0),(1,1),(2,1),(2,2),(3,2)\}$, in particular $\operatorname{t}(S)=5$. $S$ is almost symmetric, since $2\operatorname{g}(S)-\operatorname{t}(S)+1=2\cdot 8-5+1=12=|\negf+\mathbf{1}|_{\times}$.\\
In Figure~\ref{fig:GNS} we consider the GNS $S$, where the elements in $\operatorname{PF}(S)$ are marked in red and the elements $\operatorname{H}(S)\setminus \operatorname{PF}(S)$ are marked in black. The elements of $S$ lying in the red region are the elements of $\operatorname{N}(\negf)$. We next compute the set $\operatorname{C}(S,(3,1))$ and arrange it according to Theorem \ref{AsymAp}. \\
Considering the lexicographic order, we can arrange the set in the following way:\\
$\operatorname{C}(S,(3,1))=\{(0,0),(0,3),(1,3),(2,0), (2,3),(3,0),(3,3),(4,0),(4,3),(5,0),(6,0),(6,3)\} \newline \bigcup \  \{(4,1),(4,2),( 5,2),(5,3)\}$.

Considering the graded lexicographic order, we can arrange the set in the following way:\\
$\operatorname{C}(S,(3,1))=\{(0,0),(2,0), (0,3),(3,0),(1,3),(4,0),(2,3),(5,0),(3,3),(6,0),(4,3),(6,3)\}  \newline \bigcup \  \{(4,1),(4,2),( 5,2),(5,3)\}$.

It is easy to see, in both cases, that Theorem~\ref{AsymAp} is verified.

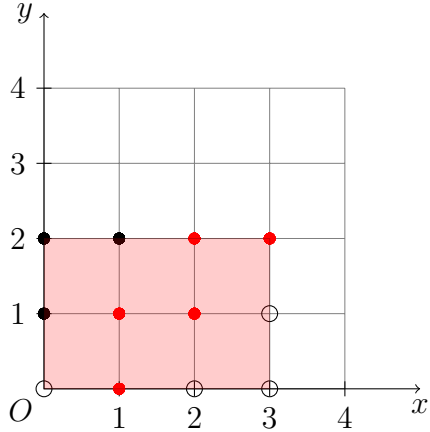
\begin{figure}[htp]

\begin{tikzpicture} 
\draw [help lines] (0,0) grid (4,4);
\draw [<->] (0,5) node [left] {$y$} -- (0,0)
-- (5,0) node [below] {$x$};
\foreach \i in {1,...,4}
\draw (\i,1mm) -- (\i,-1mm) node [below] {$\i$} 
(1mm,\i) -- (-1mm,\i) node [left] {$\i$}; 
\node [below left] at (0,0) {$O$};

\draw (0,0) circle (3pt);
\draw (2,0) circle (3pt);
\draw (3,0) circle (3pt);
\draw (3,1) circle (3pt);
\draw [mark=*, color=red] plot (1,0);
\draw [mark=*] plot (0,1);
\draw [mark=*] plot (0,2);
\draw [mark=*, color=red] plot (1,1);
\draw [mark=*] plot (1,2);
\draw [mark=*, color=red] plot (2,1);
\draw [mark=*, color=red] plot (2,2);
\draw [mark=*, color=red] plot (3,2);

\draw [fill=red, opacity=0.2] (0,0) rectangle (3,2);
\end{tikzpicture}
\caption{The generalized numerical semigroup in Example~\ref{exaf}.}\label{fig:GNS}
\end{figure}
\label{exaf}

\end{example}

The following two characterizations for symmetric and pseudo-symmetric GNS are easy consequences of Theorem~\ref{AsymAp} and Remark~\ref{remarkAp}. They were proved in a different way in \cite{phdthesis}. They can be also viewed as generalizations of well-known results on numerical semigroups (see \cite[Propositions 4.10 and 4.15]{book}).

\begin{corollary}
Let $S\subseteq \mathbb{N}_0^{d}$ be GNS, $\negn\in S$ and $\prec$ a monomial order in $\mathbb{N}_0^{d}$. Then $S$ is symmetric if and only if $\operatorname{C}(S,\negn)=\{\nega_{0}\prec\nega_{1}\prec \cdots \prec \nega_{m}\}$ with $\nega_{i}+\nega_{m-i}=\nega_{m}$, for $i=0,1,\ldots,m$.
\end{corollary}

\begin{corollary}
Let $S\subseteq \mathbb{N}_0^{d}$ be a GNS, $\negn\in S$ and $\prec$ a monomial order in $\mathbb{N}_0^{d}$. Then $S$ is pseudo-symmetric if and only if $\operatorname{C}(S,\negn)=\{\nega_{0}\prec\nega_{1}\prec \cdots \prec \nega_{m}=\negf+\negn\}\cup \left\lbrace {\negf \over 2}+\negn\right\rbrace$ where $\negf$ is maximal in $\operatorname{H}(S)$ with respect to $\leq$ and $\nega_{i}+\nega_{m-i}=\nega_{m}$, for $i=0,1,\ldots,m$.
\end{corollary}

\section{Computing all almost symmetric GNSs with fixed Frobenius element}

Starting from the observation that every $\negf \in \NN_0^d \setminus \{\mathbf{0}\}$ is the Frobenius element of an almost symmetric GNS, since we can consider the ordinary GNS $S(\negf)$ as introduced in Example~\ref{ex2}, we present in this section some procedures to obtain all almost symmetric GNSs having $\negf$ as Frobenius element. In special, we show that these methods enable to arrange all almost symmetric GNSs with a fixed Frobenius element in a rooted tree. Algorithms with this purpose are known for numerical semigroups (see \cite{giventype} and \cite{construction}).\\

We consider first when the unitary extension of a GNS preserves the property of being almost symmetric.  Recall that if $S$ is a GNS and $x \notin S$ then $S \cup \{\negx\}$ is a GNS if and only if $\negx \in \operatorname{SG}(S)$.

\begin{proposition}
	Let $S\subseteq \NN_0^d$ be an almost symmetric GNS with Frobenius element $\negf$ and $\negx \in \operatorname{SG}(S)\setminus \{\negf\}$. Then $S\cup \{\negx\}$ is almost symmetric if and only if $\operatorname{PF}(S \cup \{\negx\})=\operatorname{PF}(S)\setminus \{\negx,\negf-\negx\}$.
	\label{sonAS}
\end{proposition}

\begin{proof}
	$(\Rightarrow)$ Suppose $S\cup \{\negx\}$ is almost symmetric. Let $\negh \in \operatorname{PF}(S \cup \{\negx\})$, observe that $\negf-\negh \notin S\cup \{\negx\}$ and also that $\negh\neq \negf-\negx$ since $\negh \in \operatorname{PF}(S\cup \{\negx\})$ and $\negh+\negx=\negf\notin S\cup \{\negx\}$. Let $\negs \in S^*$, then $\negh+\negs \in S\cup \{\negx\}$, but $\negh+\negs\neq \negx$ otherwise $(\negf-\negx)+\negs=\negf-\negh$, that is a contradiction since $\negf-\negx \in \operatorname{PF}(S)$ and $\negf-\negh\notin S$, so $\negh \in \operatorname{PF}(S)$. Let $\negh \in \operatorname{PF}(S)\setminus \{\negx,\negf-\negx\}$, if $\negh+\negx\notin S\cup \{\negx\}$ then $\negf-\negh-\negx \in S\cup \{\negx\}$ or $\negf-\negh-\negx\in \operatorname{PF}(S\cup \{\negx\})$, since $S\cup \{\negx\}$ is almost symmetric. In both cases $\negf-\negh\in S\cup \{\negx\}$ and since $\negf-\negh\neq \negx$ then $\negf-\negh\in S$, but this is a contradiction, in fact $\negf-\negh\in \operatorname{PF}(S)$ since $S$ is almost symmetric. So $\negh \in \operatorname{PF}(S \cup \{\negx\})$.\\
	$(\Leftarrow)$ Suppose $\operatorname{PF}(S \cup \{\negx\})=\operatorname{PF}(S)\setminus \{\negx,\negf-\negx\}$. Then $\operatorname{t}(S \cup \{\negx\})=\operatorname{t}(S)-2$ and $2\operatorname{g}(S\cup \{\negx\})-\operatorname{t}(S\cup \{\negx\})+1=2\operatorname{g}(S)-\operatorname{t}(S)+1=|\negf+\mathbf{1}|_{\times}$.
\end{proof}

In the following result it is shown when the property of being almost symmetric is preserved by removing an element from a GNS. Recall that if $S$ is a GNS and $\negx \in S$ then $S\setminus \{\negx\}$ is a GNS if and only if $\negx$ is a minimal generators of $S$, that is, $\negx \in S^*\setminus (S^* +S^*)$.

\begin{proposition}
	Let $T\subseteq \NN_0^d$ be an almost symmetric GNS with Frobenius element $\negf$ and $\negx $ be a minimal generator of $T$ such that $\negx< \negf$. Then $T\setminus \{\negx\}$ is almost symmetric if and only if $\negf-\negx\in \operatorname{PF}(T \setminus \{\negx\})$, in particular  $\operatorname{PF}(T)=\operatorname{PF}(T\setminus \{\negx\})\setminus \{\negx,\negf-\negx\}$.
	\label{parentAS}
\end{proposition}

\begin{proof}
	If $T\setminus \{\negx\}$ is almost symmetric, since $\negx$ is a minimal generator of $T$ we have $\negx \in \operatorname{PF}(T\setminus \{\negx\})$, and in particular $\negf-\negx \in \operatorname{PF}(T\setminus \{\negx\})$. Suppose that $\negf-\negx\in \operatorname{PF}(T \setminus \{\negx\})$, we prove that $\operatorname{PF}(T)=\operatorname{PF}(T\setminus \{\negx\})\setminus \{\negx,\negf-\negx\}$. Let $\negh \in \operatorname{PF}(T)$. Then $\negf-\negh \in \operatorname{PF}(T)$, so $\negh\neq \negx$ and $\negh\neq \negf-\negx$. If $\negt \in T\setminus \{\negx\}\subset T$ then $\negh+\negt\in T$ and if $\negh+\negt=\negx$ then $\negf-\negh=(\negf-\negx)+\negt \in T\setminus \{\negx\}\subset T$ since $\negf-\negx\in \operatorname{PF}(T \setminus \{\negx\})$, that is a contradiction. So $\negh+\negt\in T\setminus \{\negx\}$ and $\negh\in \operatorname{PF}(T\setminus \{\negx\})\setminus \{\negx,\negf-\negx\}$. Suppose that $\negh\in \operatorname{PF}(T\setminus \{\negx\})\setminus \{\negx,\negf-\negx\}$, in order to prove that $\negh \in \operatorname{PF}(T)$ it suffices to show that $\negh+\negx\in T$. If $\negh+\negx\notin T$ then $\negf-\negh-\negx \in T$ or $\negf-\negh-\negx \in \operatorname{PF}(T)$. In both cases $\negf-\negh\in T\setminus \{\negx\}$, that is a contradiction with $\negh\in \operatorname{PF}(T\setminus \{\negx\})$. So $\operatorname{PF}(T)=\operatorname{PF}(T\setminus \{\negx\})\setminus \{\negx,\negf-\negx\}$, in particular $\operatorname{t}(T \setminus \{\negx\})=\operatorname{t}(T)+2$ and we have $2\operatorname{g}(T \setminus \{\negx\})-\operatorname{t}(T \setminus \{\negx\})+1=2\operatorname{g}(T)-\operatorname{t}(T)+1=|\negf+\mathbf{1}|_{\times}$, that is $T \setminus \{\negx\}$ is almost symmetric.
\end{proof}

\begin{remark} Observe that the methods of adding and removing single elements in almost symmetric GNSs described respectively in Propositions \ref{sonAS} and \ref{parentAS} are somewhat inverse to each other in the sense that:
	\begin{itemize}
		\item if $\negx \in \operatorname{SG}(S)\setminus \{\negf\}$ then $\negx$ is a minimal generator of $S\cup \{\negx\}$ with $\negx<\negf$;  and
		\item if $\negx$ is a minimal generator of $T$ with $\negx<\negf$ then $\negx\in \operatorname{SG}(T\setminus \{\negx\})\setminus \{\negf\}$.
	\end{itemize}
\end{remark} 

\begin{definition}
	Let $S\subseteq \NN_0^d$ be an almost symmetric GNS with Frobenius element $\negf$ and $\prec$ be a monomial order. Considering $\operatorname{N}(S)=\{\negx \in S\mid \negx \leq \negf\}$, we define:
	\begin{itemize}
		\item[(a)] $\operatorname{D}(S)=\{\negx \in \operatorname{SG}(S)\setminus \{\negf\}\mid S\cup \{\negx\}\ \mbox{is almost symmetric}\}$;
		\item[(b)] $\operatorname{D}_{\prec}(S)=\{\negx \in \operatorname{D}(S)\mid \negx \prec \negy\ \mbox{for all}\ \negy \in (\operatorname{N}(S)\setminus \{\mathbf{0}\})\cup \{\negf\}\}$.
	\end{itemize}
	Moreover, if $S$ is not ordinary we define $\operatorname{low}_{\prec}(S)=\min_{\prec}(\operatorname{N}(S)\setminus \{\mathbf{0}\})$.
\end{definition} 

For a nonordinary GNS, the existence of an element satisfying the conditions of the Proposition~\ref{parentAS} is guaranteed by the following:

\begin{lemma}
	Let $T\subseteq \NN_0^d$ be a nonordinary almost symmetric GNS with Frobenius element $\negf$. Let $\prec$ be a monomial order and $\negx=\operatorname{low}_{\prec}(T)$. Then $\negf-\negx \in \operatorname{PF}(T\setminus \{\negx\})$, in particular $T\setminus\{\negx\}$ is almost symmetric.
	\label{xmin}
\end{lemma}

\begin{proof}
	Since $T$ is not ordinary, the set $\{\negz \in T^*\mid \negz \leq \negf\}$ is not empty, so $\negx$ is well defined. Moreover $\negx$ is a minimal generator of $T$, otherwise $\negx=\negx_1+\negx_2$ with $\negx_1,\negx_2\in T^*$ and $\negx_1\prec\negx$, so $T\setminus \{\negx\}$ is a GNS. Let $\negy=\negf-\negx$, we prove that $\negy \in \operatorname{PF}(T\setminus \{\negx\})$. Let $\negt \in T\setminus \{\negx\}$ and observe that $\negf$ is the Frobenius element of $T\setminus \{\negx\}$. If $\negt \nleq \negf$ then $\negy+\negt\nleq \negf$, so $\negy+\negt \in T\setminus \{\negx\}$. If $\negt\leq \negf$ then $\negx\prec \negt$ and $\negf=\negy+\negx\prec \negy+\negt$, that is, $\negy+\negt \in T\setminus \{\negx\}$. Finally $T\setminus\{\negx\}$ is almost symmetric by Proposition~\ref{parentAS}.
\end{proof}


\begin{lemma}
	Let $S\subseteq \NN_0^d$ be a nonordinary almost symmetric GNS with Frobenius element $\negf$. Then there exists a sequence $S_{n}\subset S_{n-1}\subset\cdots \subset S_{1}$ of almost symmetric GNSs such that:
	\begin{itemize}
		\item $S_{n}=S(\negf)$
		\item $S_1=S$
		\item $S_{i+1}=S_{i}\setminus \{\operatorname{low}_{\prec}(S_i)\}$ for every $i \in \{1,\ldots,n-1\}$, in particular $\operatorname{low}_{\prec}(S_i)\in \operatorname{D}(S_{i+1})$.
	\end{itemize}
	\label{path}
\end{lemma}

\begin{proof}
	Let $\prec$ be a fixed monomial order. Consider $S_{1}=S$, by Lemma~\ref{xmin} if $\negx =\operatorname{low}_{\prec}(S)$ then $S\setminus \{\negx\}$ is an almost symmetric GNS with Frobenius element $\negf$. Put $S_{2}=S\setminus \{\negx\}$, that is $S_{2}\cup \{\negx\}=S_{1}$. Obviously $\negx \in \operatorname{D}(S_2)$. If $S_2$ is ordinary then we conclude, otherwise we can repeat the procedure with $S_i$, for $i\geq 2$, obtaining an almost symmetric GNS $S_{i+1}$ such that $S_{i+1}=S_{i}\setminus \{\operatorname{low}_{\prec}(S_i)\}$. The procedure stops when we obtain $S_{i}=S(\negf)$ for some $i$.
\end{proof}

For $\negf \in \NN_0^d$, we define $\mathcal{A}(\negf)$ be the set of all almost symmetric GNSs having Frobenius element $\negf$. We can define the graph $\mathcal{G}(\negf)=(\mathcal{A}(\negf), \mathcal{E})$ whose set of vertices is $\mathcal{A}(\negf)$ and $\mathcal{E}$ is the set of edges, where $(T,S)\in\mathcal{E}$ if $T=S\cup \{\negx\}$ for some $\negx \in \operatorname{SG}(S)\setminus \{\negf\}$ such that $T$ is almost symmetric.
By Lemma~\ref{path} from every almost symmetric GNS there exists a path of edges in $\mathcal{G}(\negf)$ linking it to the ordinary GNS $S(\negf)$. 
So we can consider a first procedure to generate all almost symmetric GNS with fixed Frobenius element $\negf$. Starting from $T=S(\negf)$, we can consider the almost symmetric GNSs $T_{\negx}=T\cup \{\negx\}$ for each $\negx \in \operatorname{D}(T)$, then we repeat the procedure to each semigroup $T_{\negx}$, and so on. The procedure stops when an irreducible GNS is obtained since, according to Proposition \ref{maxirreducible}, they are maximal with respect to inclusion among the Frobenius GNSs with a fixed Frobenius element. Hence all almost symmetric are produced by the procedure. Observe however that, building $\mathcal{G}(\negf)$ by this method, some redundancies can be produced, which means that some almost symmetric GNSs can be obtained more than one time, as in the following example:

\begin{example}
	Let $\mathbf{f}=(2,1)$ and $S(\negf)=\NN_0^2\setminus \{(0,1),(1,0),(1,1),(2,0),(2,1)\}$.\\
	Following Proposition~\ref{sonAS}, from $S(\negf)$ we obtain the following almost symmetric GNS:
	\begin{itemize}
		\item $S_1^{(1)}=S(\negf) \cup \{(2,0)\}=\NN_0^2\setminus \{(0,1),(1,0),(1,1),(2,1)\}$.
		\item $S_1^{(2)}=S(\negf) \cup \{(1,1)\}=\NN_0^2\setminus \{(0,1),(1,0),(2,0),(2,1)\}$.
	\end{itemize}
	Applying again Proposition~\ref{sonAS} in the above GNS we obtain:
	\begin{itemize}
		\item $S_2^{(1)}=S_1^{(1)} \cup \{(1,1)\}=S_1^{(2)} \cup \{(2,0)\}=\NN_0^2\setminus \{(0,1),(1,0),(2,1)\}$.
		\item $S_2^{(2)}=S_1^{(1)} \cup \{(1,0)\}=\NN_0^2\setminus \{(0,1),(1,1),(2,1)\}$.
		\item $S_2^{(3)}=S_1^{(2)} \cup \{(0,1)\}=\NN_0^2\setminus \{(1,0),(2,0),(2,1)\}$.
	\end{itemize}
	The three semigroups above are all irreducible (they are symmetric in particular). In Figure~\ref{graph} the graph $\mathcal{G}((2,1))$ is pictured.
	
	\begin{figure}[h!]
		\includegraphics[scale=0.7]{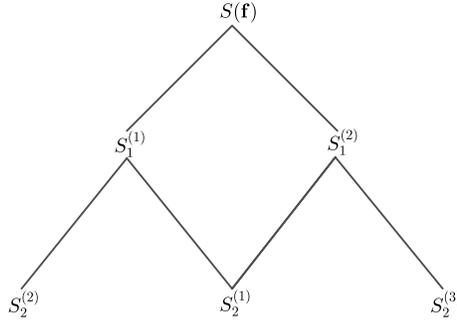}
		\caption{The graph $\mathcal{G}(\negf)$ of all almost symmetric GNS with Frobenius element $(2,1)$.}
		\label{graph}
	\end{figure}
	
	In particular we obtain the following different chains of almost symmetric GNS with Frobenius element $\negf=(2,1)$:
	\begin{itemize}
		\item $S(\negf) \subset S_1^{(1)}\subset S_2^{(2)}$.
		\item $S(\negf) \subset S_1^{(1)}\subset S_2^{(1)}$.
		\item $S(\negf) \subset S_1^{(2)}\subset S_2^{(1)}$.
		\item $S(\negf) \subset S_1^{(2)}\subset S_2^{(3)}$.
	\end{itemize}
	\label{exaGraph}
\end{example}

Notice that, starting from an almost symmetric GNS $(S,\mathbf{f})$, the aforementioned procedure to obtain almost symmetric GNSs by adding special gaps always reaches an irreducible GNS $T$ having Frobenius element $\mathbf{f}$. In this case, $T=S\cup A$, where $A$ is a set of pseudo-Frobenius elements of $S$ satisfying 
$$\operatorname{PF}(S)=\operatorname{PF}(T) \cup \{\mathbf{x}, \mathbf{f}-\mathbf{x} \mid  \mathbf{x}\in A\}.$$ Conversely, if $S=T\setminus A$ for $T$ an irreducible GNS and $A$ a subset of $T$ satisfying the above relation, then $S$ is necessarily an almost symmetric GNS (by counting the genus and the type of the semigroups involved, that is, by Definition~\ref{ASdef}). Moreover, starting from an almost symmetric GNS $S$, it is possible to take an irreducible GNS $T$ with $S=T\setminus A$ in such a way that the elements in $A$ are actually minimal generators of $T$. In particular, we obtain the following:

\begin{theorem} \label{asirred} Let $(S,\mathbf{f})$ be a Frobenius GNS in $\NN_0^d$. Then $S$ is almost symmetric if and only if there exists an irreducible GNS $T$ having Frobenius element $\negf$ such that $S=T\backslash A$, where $A$ is a subset of minimal generators of $T$ satisfying 
	$$\operatorname{PF}(S)=\operatorname{PF}(T) \cup \{\mathbf{x}, \mathbf{f}-\mathbf{x} \mid \mathbf{x}\in A\}.$$ In this case, $\operatorname{t}(S)=2|A|+\operatorname{t}(T)$.
\end{theorem}
\begin{proof}
	$(\Leftarrow)$ It follows immediately from Definition~\ref{ASdef}.\\
	$(\Rightarrow)$ Let $S$ be an almost symmetric GNS with Frobenius element $\negf$. Hence, by Proposition~\ref{bigprop}, if $\prec$ is a monomial order in $\NN_0^d$, depending on the parity of $\operatorname{t}(S)$, we can decompose $\operatorname{PF}(S)$ in one of the following forms 
	$$\{\negf \succ \negh_1\succ\cdots\succ\negh_r\succ \negf-\negh_r\succ \cdots\succ \negf-\negh_1\}$$
	and 
	$$\{\negf \succ \negh_1\succ\cdots\succ\negh_r\succ \negf/2\succ \negf-\negh_r\succ \cdots\succ \negf-\negh_1\}.$$
	Hence, consider $A=\{\negh_i \mid i=1,\ldots,r\}$. Notice that for all $i,j\in \{1,\ldots,r\}$, we have that $\negh_i+\negh_j\in S$, in fact $\negh_i+\negh_j\succ (\negf-\negh_j)+\negh_j=\negf$.
	Furthermore, we have $\operatorname{t}(S)=2|A|+e$, where $e=1,2$. In particular $e=1$ in the first case, $e=2$ in the second case.\\
	Observe that $T=S\cup A$ is a Frobenius GNS with Frobenius element $\negf$. In fact, by the above observation and considering that $A\subset \operatorname{PF}(S)$, if $\negx,\negy\in T$ then $\negx+\negy\in S\subseteq T$. Moreover the complement of $T$ in $\NN^d_0$ is finite since is contained in $\operatorname{H}(S)$. Let us show that $A$ is a set of minimal generators of $T$. If $\negh_i=\negx+\negy$, for $\negx, \negy\in T^*$, then we have only the following possibilities:
	\begin{itemize}
		\item $\negx, \negy\in S$ that implies $\negh_i\in S$, contradicting $\negh_i\in \operatorname{PF}(S)$;
		\item $\negx, \negy\in A$ that implies $\negh_i\in S$, contradicting again $\negh_i\in \operatorname{PF}(S)$;
		\item $\negx\in A$ and $\negy\in S$, and since $\negh_i\in \operatorname{H}(S)$, we have $\negx\not\in \operatorname{PF}(S)$ and so $\negx\not \in A$.
	\end{itemize}
	So it remains to prove that $T$ is irreducible. As $S$ is almost symmetric, we have $2 \operatorname{g}(S)+1-\operatorname{t}(S)=|\negf+\mathbf{1}|_\times.$ Hence, since $\operatorname{g}(T)=\operatorname{g}(S)-|A|$, we obtain $$2 \operatorname{g}(T)+1-e=2 \operatorname{g}(T)+1-(\operatorname{t}(S)-2|A|)=2 \operatorname{g}(S)+1- \operatorname{t}(S)=|\negf+\mathbf{1}|_\times,$$
	and as $2 \operatorname{g}(T)+1- \operatorname{t}(T)\geq |\negf+\mathbf{1}|_\times$, it follows that $1\leq \operatorname{t}(T)\leq e$. If $e=1$ then $\operatorname{PF}(T)=\{\negf\}$, so $T$ is irreducible. If $e=2$ observe that $\{\negf,\negf/2\}\subseteq \operatorname{PF}(T)$, and since $\operatorname{t}(T)\leq 2$ we have $\operatorname{PF}(T)=\{\negf,\negf/2\}$. In particular, $T$ is irreducible with $\operatorname{t}(T)=\operatorname{t}(S)-2|A|$. It is easy to see that in each case $\operatorname{PF}(S)=\operatorname{PF}(T) \cup \{\mathbf{x}, \mathbf{f}-\mathbf{x} \mid \mathbf{x}\in A\}$.
\end{proof}

The condition  $\operatorname{PF}(S)=\operatorname{PF}(T) \cup \{\mathbf{x}, \mathbf{f}-\mathbf{x} \mid \mathbf{x}\in A\}$ in the above theorem can be formulated into a more computational way as follows, which makes it simpler to be verified. This fact allows us to bring to the GNSs an extension of the description given in \cite{construction} for almost symmetric numerical semigroups regarding irreducible numerical semigroups.

\begin{proposition}
Let $(T,\negf)$ be an irreducible GNS and $A$ a subset of minimal generators of $T$ smaller than $\negf$ with respect to partial order $\leq$. Considering the Frobenius GNS $S=T\setminus A$, then $\operatorname{PF}(S)=\operatorname{PF}(T) \cup \{\mathbf{x}, \mathbf{f}-\mathbf{x} \mid \mathbf{x}\in A\}$ if and only if $\negx+\negy-\negf\notin S$  for all $\negx,\negy \in A$.
\end{proposition}
\begin{proof} Given $\negx, \negy\in A$, it follows by assumption that $\negx,\negf-\negy\in \operatorname{PF}(S)$ and consequently $\negx+\negy-\negf=\negx-(\negf-\negy)\not\in S$. Conversely,  note that $A \subseteq \operatorname{PF}(S)$, since for any $\negs\in S^*$ and $\negx\in A$, we obtain $\negx+\negs\in T$ but $\negx+\negs\not\in A$. Furthermore $\operatorname{PF}(T)\subseteq \operatorname{PF}(S)$, in fact it is trivial $\negf \in \operatorname{PF}(S)$ and if $\negf/2+\negs \notin S$ for some $s\in S^*$ then $\negf/2+\negs=\negx$ for some $\negx \in A$, but in this case $2\negs=2x-\negf \in S$, that is a contradiction.	Now, as $\operatorname{PF}(T)\cup A\subseteq \operatorname{PF}(S)$, if $\negy\in \operatorname{PF}(S)\backslash (\operatorname{PF}(T)\cup A)$, being $T$ irreducible, then $\negf-\negy\in T$ by Proposition \ref{differenceFrobenius}. But $\negf-\negy\in \operatorname{H}(S)$ since $\negy\in \operatorname{PF}(S)$. So, $\negf-\negy\in T\backslash S=A$. Therefore, $\negf-\negy=\negx$ for some $\negx\in A$, and then $\negy=\negf-\negx$, which proves that $ \operatorname{PF}(S)\subseteq  \operatorname{PF}(T) \cup \{\mathbf{x}, \mathbf{f}-\mathbf{x} \mid \mathbf{x}\in A\}$. Hence, it suffices to prove that $\negf-\negx\in \operatorname{PF}(S)$ for all $\negx\in A$. 
Observe that $\negf-\negx\in \operatorname{H}(T)$ and thus $\negf-\negx\in  \operatorname{H}(S)$. Now, assume that $\negf-\negx+\negs\in \operatorname{H}(S)$ for some $\negs\in S^*$. If  $\negf-\negx+\negs\in T$, then $\negf-\negx+\negs=\negy\in A$ and so $\negx+\negy-\negf\in S$, which is a contradiction. Hence, $\negf-\negx+\negs\in \operatorname{H}(T)$. Being $T$ irreducible, again by Proposition~\ref{differenceFrobenius}, we have either $\negf-\negx+\negs=\negf/2$ or $\negx-\negs=\negf-(\negf-\negx+\negs)\in T$. The first relation implies that $2\negx-\negf=2(\negx-\negf/2)=2(\negf+\negs)\in S$, contrary to the assumption. As $\negx$ is a minimal generator of $T$, the second one implies that $\negs=\negx$, which is not possible because $S=T\setminus A$. Therefore, $\{\negf-\negx \mid \negx\in A\}\subseteq \operatorname{PF}(S)$, which establishes that $\operatorname{PF}(S)=\operatorname{PF}(T) \cup \{\mathbf{x}, \mathbf{f}-\mathbf{x} \mid \mathbf{x}\in A\}$.
\end{proof}

In order to have a procedure that generates all almost symmetric GNS without redundancies we can define a subgraph of $\mathcal{G}(\negf)$ that is a \emph{tree}.

\begin{definition}
	Let $\prec$ be a monomial order, $\negf\in \NN_0^d$ and $\mathcal{A}(\negf)$ be the set of all almost symmetric GNS having Frobenius element $\negf$.\\ We define $\mathcal{G}_{\prec}(\negf)=(\mathcal{A}(\negf), \mathcal{E}_{\prec})$ as the graph whose set of vertices is $\mathcal{A}(\negf)$ and the set of edges is  $\mathcal{E}_{\prec}$, where $(T,S)\in\mathcal{E}_{\prec}$ if $S=T\setminus \{\operatorname{low}_{\prec}(T)\}$.
	If $(T,S)\in \mathcal{E}_{\prec}$ we say that $T$ is a \emph{child} of $S$. 
	\label{defTree}
\end{definition}

Observe that $\mathcal{G}_{\prec}(\negf)$ is a subgraph of $\mathcal{G}(\negf)$, in fact: if $(T,S)\in \mathcal{E}_{\prec}$ then $S=T\setminus \{\operatorname{low}_{\prec}(T)\}$, that is, $S\cup \{\operatorname{low}_{\prec}(T)\}=T$, and so $(T,S)\in \mathcal{E}$. 

\begin{theorem}
	Let $\prec$ be a monomial order and $\negf\in \NN_0^d$. Then $\mathcal{G}_{\prec}(\negf)$ is a rooted tree whose root is $S(\negf)$. Moreover if $S\in \mathcal{A}(\negf)$, all the children of $S$ are the semigroups $S\cup \{\negx\}$ for all $\negx \in \operatorname{D}_{\prec}(S)$.
	\label{thmTree}
\end{theorem}
\begin{proof}
	Let $S\in \mathcal{A}(\negf)$. As in the proof of Lemma~\ref{path} we can produce a chain of semigroups $S_{n}\subset S_{n-1}\subset\cdots \subset S_{1}$ such that $S_1=S$, $S_n=S(\negf)$ and $S_{i+1}=S_{i}\setminus \{\operatorname{low}_{\prec}(S_i)\}$. In particular $(S_1,S_2),(S_2,S_3),\ldots,(S_{n-1},S_n)$ is a path of edges of $\mathcal{G}_{\prec}(\negf)$ from $S$ to $S(\negf)$. If there exists another path, then for some $i\in \{1,\ldots,n\}$ there exist two different semigroups $T_1,T_2\in \mathcal{A}(\negf)$ such that $(S_i,T_1),(S_i,T_2)\in \mathcal{E}_{\prec}$, in particular $T_1=S\setminus \{\operatorname{low}_{\prec}(S_i)\}=T_2$, that is a contradiction. So we conclude that $\mathcal{G}_{\prec}(\negf)$ is a rooted tree whose root is $S(\negf)$. If $T$ is a child of $S$ then $S=T\setminus\{\operatorname{low}_{\prec}(T)\}$, that is $T= S\cup \{\operatorname{low}_{\prec}(T)\}$. In particular $\operatorname{low}_{\prec}(T)\in \operatorname{D}(S)$, and if $\negy \in \operatorname{N}(S)\cup \{\negf\}$ then $\negy\in \operatorname{N}(T)\cup \{\negf\}$, so $\operatorname{low}_{\prec}(T)\prec \negy$. Therefore $\operatorname{low}_{\prec}(T)\in \operatorname{D}_{\prec}(S)$ and this proves the second statement.
\end{proof}

Let $\prec$ be a monomial order in $\NN_0^d$. We define $\mathcal{F}(S)=\{S\cup \{\negx\}\mid \negx \in \operatorname{D}_{\prec}(S)\}$. By the above theorem, in order to produce all almost symmetric GNS with fixed Frobenius element $\negf$ we can consider the following steps:
\begin{enumerate}
	\item $V_1=\{S(\negf)\}$
	\item $V_{i}=\bigcup_{S\in V_{i-1}}\mathcal{F}(S)$
\end{enumerate}
Observe that for every $T\in \mathcal{F}(S)$ we have $\operatorname{g}(T)=\operatorname{g}(S)-1$ and $\operatorname{t}(T)=\operatorname{t}(S)-2$. Furthermore, the procedure stops when all semigroups in $V_{i}$ are irreducible.\\
In particular, if $\mathcal{I}(\negf)$ is the set of all irreducible GNS with Frobenius element $\negf$, from Theorem~\ref{irreducible} we have
$$\mathcal{I}(\negf)=V_{s} \ \mbox{and} \ \mathcal{A}(\negf)=\bigcup_{i=1}^{s}V_i, \ \mbox{where} \ s=\left\{\begin{array}{cccl}
\frac{|\negf+\mathbf{1}|_{\times}}{2} & , & \mbox{if} & \negf/2\notin \NN_0^d \\
\frac{|\negf+\mathbf{1}|_{\times}+1}{2} & , & \mbox{if} & \negf/2\in \NN_0^d
\end{array}\right..$$

\begin{example}
	Let $\negf=(2,1)$. If we consider $\prec_{lex}$ be the lexicographic order, we can produce the tree $\mathcal{G}_{\prec_{lex}}(\negf)$ as in Example~\ref{exaGraph} without considering the redundant computation $S_1^{(2)} \cup \{(2,0)\}=\NN_0^2\setminus \{(0,1),(1,0),(2,1)\}$.
	In fact $S_1^{(2)}=\NN_0^2\setminus \{(0,1),(1,0),(2,0),(2,1)\}$, in particular $\operatorname{N}(S_1^{(2)})=\{(0,0),(1,1)\}$ and $(1,1)\prec_{lex} (2,0)$.\\
	If we consider $\prec_{rlex}$ be the reverse lexicographic order, then the irredundant computation in Example~\ref{exaGraph} is $S_1^{(1)} \cup \{(1,1)\}=\NN_0^2\setminus \{(0,1),(1,0),(2,1)\}$. In fact $S_1^{(1)}=\NN_0^2\setminus \{(0,1),(1,0),(1,1),(2,1)\}$, in particular $\operatorname{N}(S_1^{(1)})=\{(0,0),(2,0)\}$ and $(2,0)\prec_{rlex} (1,1)$.\\
	In particular, different monomial orders can arrange the set $\mathcal{A}(\negf)$ in different trees $\mathcal{G}_{\prec}(\negf)$.
	\begin{figure}[h!]
		\subfloat[][$\mathcal{G}_{\prec}((2,1))$ with $\prec$ lex order]{\includegraphics[width=.45\textwidth]{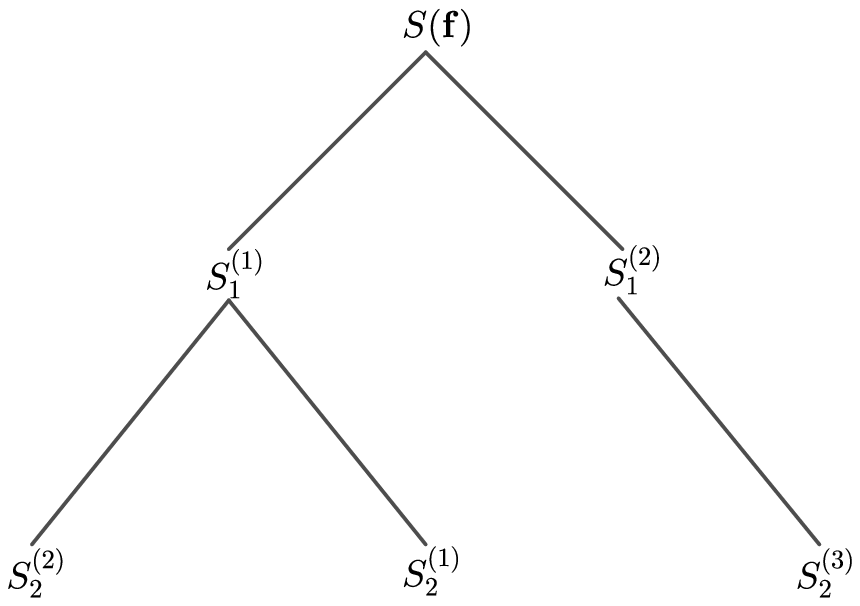}} \quad\subfloat[][$\mathcal{G}_{\prec}((2,1))$ with $\prec$ revlex order]{\includegraphics[width=.45\textwidth]{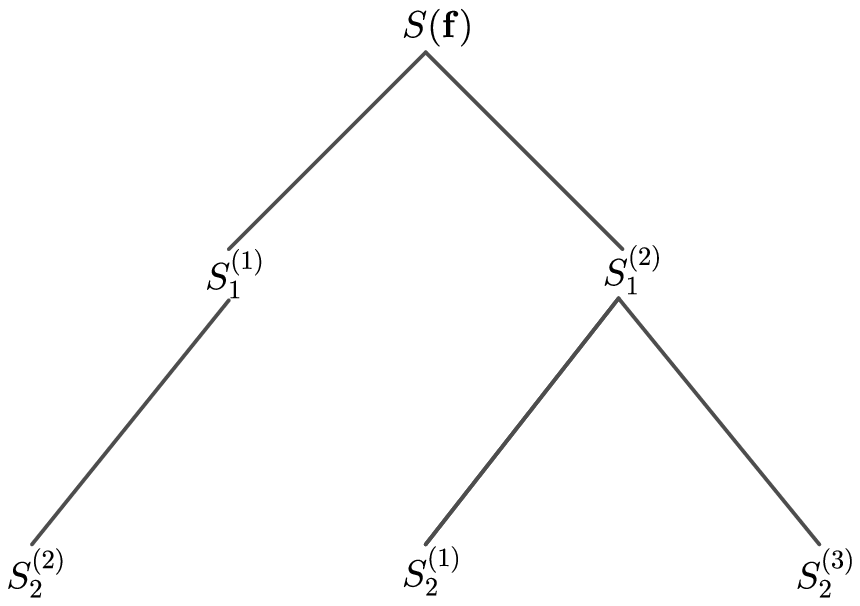}} 
		\caption{Example of $\mathcal{G}_{\prec}((2,1))$ with two different monomial orders}
	\end{figure}
\end{example}

\section{Concluding remarks}

We introduced an extension to the notion of almost-symmetry in the setting of GNSs and we deduced many results and properties for this new class of GNSs, generalizing the study made in numerical semigroups and also the approach presented in \cite{irreducible} for irreducible GNSs. In particular we provided a comprehension on how almost symmetric GNSs can be are organized. Many papers are devoted to study the role played by the almost symmetric numerical semigroups in the theory of numerical semigroups and this work intends to contribute so that the study of GNSs can also move in this direction. Among all, we were inspired by \cite{symmetries}, in which other interesting properties are contained, but we could also mention \cite{giventype}, \cite{higher} and \cite{construction}. We list below some other open questions concerning GNSs:
	\begin{itemize}
		\item Apart from the almost symmetric GNSs, are there other families of GNSs satifying the equality of sets $\mathrm{Maximals}_{\leq_{S}}\operatorname{Ap}(S,\negn)=\mathrm{Maximals}_{\leq_{S}}\operatorname{C}(S,\negn)$? 
		\item Generalizing a well-known conjecture for numerical semigroups (introduced in \cite{WilfNumerical}, namely Wilf's conjecture), it is proposed in \cite{wilf} the generalized Wilf's conjecture for GNSs that, for a Frobenius GNS $(S, \negf)$ in $\NN_0^d$, asks whether it holds that
		$$e(S)n(S)\geq d |\negf+\textbf{1}|_\times \, ,$$
		where $e(S)$ is the embedding dimension of $S$ and $$n(S)=|\{\negx\in S \mid \negx\leq \negh \ \mbox{for some } \negh\in \operatorname{H}(S)\}|.$$
		We mention that a generalization of Wilf's conjecture was introduced first in \cite{AnExtension} for a larger class of semigroups than GNSs. The conjecture introduced in \cite{wilf} is different and it is defined only for GNSs, but it is shown that, for a GNS, if it is true the conjecture in \cite{wilf} then it is true also the conjecture in \cite{AnExtension}.\\
		It is known that every almost symmetric numerical semigroup satisfies the conjecture, whose proof is due to M. La Valle and is contained in \cite{Barucci}. Moreover, it is shown in \cite{wilf} that the conjecture is verified for both classes of ordinary and irreducible GNSs. Is it true that all almost symmetric GNSs satisfy the generalized Wilf's conjecture? 
\end{itemize}

\textbf{Acknowledgements}. The authors wish to thank the referee for the valuable comments and suggestions that improved the previous version of this work.

\end{document}